\documentclass[10pt,a4paper]{amsart}
\usepackage[utf8]{inputenc}

\usepackage{amscd,amssymb,amsopn,amsmath,amsthm,graphics,amsfonts,enumerate,verbatim,calc}
\usepackage[dvips]{graphicx}

\usepackage{mathpazo}
\usepackage{color}
\usepackage{latexsym}
\usepackage{amsthm,amsfonts,amssymb,mathrsfs}
\usepackage{rotating}
\usepackage[leqno]{amsmath}
\usepackage{xspace}
\usepackage[all]{xy}
\usepackage{longtable}

\headsep=1cm \numberwithin{equation}{section}
\newtheorem{theorem}{Theorem}[section]
\newtheorem{lemma}[theorem]{Lemma}

\newtheorem{proposition}[theorem]{Proposition}
\newtheorem{corollary}[theorem]{Corollary}

\theoremstyle{definition}

\theoremstyle{remark}
\newtheorem{remark}[theorem]{Remark}
\newtheorem{fact}[theorem]{Fact}
\newtheorem{example}[theorem]{Example}
\newtheorem{observation}[theorem]{Observation}
\newtheorem{discussion}[theorem]{Discussion}
\newtheorem{question}[theorem]{Question}

\newcommand{\D}{\operatorname{\mathcal{D}}}
\newcommand{\Cl}{\operatorname{Cl}}

\newcommand{\Ass}{\operatorname{Ass}}
\newcommand{\im}{\operatorname{im}}

\newcommand{\Spec}{\operatorname{Spec}}
\newcommand{\spec}{\operatorname{spec}}
\newcommand{\Mod}{\operatorname{mod}}

\newcommand{\rad}{\operatorname{rad}}
\newcommand{\red}{\operatorname{red}}

\newcommand{\Ht}{\operatorname{ht}}
\newcommand{\f}{\operatorname{f}}
\newcommand{\pd}{\operatorname{p.dim}}

\newcommand{\fgrade}{\operatorname{fgrade}}
\newcommand{\proj}{\operatorname{proj}}

\newcommand{\UFD}{\operatorname{UFD}}
\newcommand{\K}{\operatorname{\mathcal{K}}}
\newcommand{\G}{\operatorname{G}}
\newcommand{\Gg}{\operatorname{G1}}
\newcommand{\GGg}{\operatorname{G3}}
\newcommand{\HH}{\operatorname{H}}
\newcommand{\nil}{\operatorname{nil}}

\newcommand{\V}{\operatorname{V}}

\newcommand{\id}{\operatorname{id}}
\newcommand{\Ext}{\operatorname{Ext}}
\newcommand{\Se}{\operatorname{S}}

\newcommand{\cd}{\operatorname{cd}}
\newcommand{\R}{\operatorname{R}}
\newcommand{\Supp}{\operatorname{Supp}}

\newcommand{\Hom}{\operatorname{Hom}}

\newcommand{\Proj}{\operatorname{Proj}}

\newcommand{\depth}{\operatorname{depth}}

\newcommand{\DD}{\operatorname{D}}
\newcommand{\EE}{\operatorname{E}}

\newcommand{\Char}{\operatorname{char}}
\newcommand{\coker}{\operatorname{coker}}

\newcommand{\vpl}{\operatornamewithlimits{\varprojlim}}

\newcommand{\lo}{\longrightarrow}
\newcommand{\fm}{\mathfrak{m}}
\newcommand{\fp}{\frak{p}}
\newcommand{\fq}{\frak{q}}
\newcommand{\fa}{\frak{a}}
\newcommand{\fb}{\frak{b}}

\usepackage{hyperref}

\begin{document}

\author[]{mohsen asgharzadeh}

\address{}
\email{mohsenasgharzadeh@gmail.com}

\title[ ]
{	Algebrization of some complete modules}


\begin{abstract}
	Let $(R,\mathfrak{m})$ be a Noetherian local ring and $\widehat{R}$ its $\mathfrak{m}$-adic completion. 
	We study the problem of determining when a finitely generated $\widehat{R}$-module arises from an $R$-module, i.e., when it is \emph{algebraic}. 
	We introduce and investigate the class of \emph{strongly algebraic} modules, those complete modules all of whose direct summands are algebraic.
	Our approach unifies and extends several known results of Levy--Odenthal, Weston, Peskine--Szpiro, Puthenpurakal, and several others, and provides new examples and homological criteria for algebrization.
	Applications include a computation of the Grothendieck group $G_0(R)$ in dimension one and new algebrization results for generalized Cohen--Macaulay modules and vector bundles along with a connection to local cohomology 
 modules.
\end{abstract}

\maketitle

\section{Introduction}

In this paper, $(R,\fm)$ denotes a commutative Noetherian local ring. By \emph{completion} we mean completion with respect to the $\fm$-adic topology. 
The categories of finitely generated modules over $R$ and over $\widehat{R}$ are, in general, different. 
For instance, there exist situations in which $\Mod(R)$ is countable while $\Mod(\widehat{R})$ is uncountable. 
Our aim is to understand $\Mod(\widehat{R})$ via certain data from $\Mod(R)$ and to present some applications. 
These questions have been investigated by several authors. 
In what follows, we survey some of the corresponding results, complement them, and also present new observations. 
To this end, we begin by recalling a few definitions.

By a \emph{complete module} we mean a finitely generated $\widehat{R}$-module. 
A complete module $\mathcal{M}$ is called \emph{algebraic} if there exists a finitely generated $R$-module $M$ such that $\widehat{M} \simeq \mathcal{M}$. 
Sometimes, such a module $\mathcal{M}$ is said to be \emph{extended from} a finitely generated $R$-module (see, for instance, \cite{sean}). 
A complete module such that each of its direct summands is algebraic is called \emph{strongly algebraic}. 
This is a new notion.

\begin{observation}
	Let $R$ be a local domain whose integral closure $\overline{R}$ is a finite $R$-module but not local. 
	If $R$ has an isolated singularity, then $(\overline{R})^{\widehat{}}$ is algebraic but not strongly algebraic.
\end{observation}

Let $R$ be a $1$-dimensional local ring with algebraically closed residue field such that $R_{\mathrm{red}}$ is analytically irreducible. 
Algebrization allows us to show that 
\(
\G_0 (R) \simeq \bigoplus_{\min(R)} \mathbb{Z}.
\)
For more details, see Corollary~\ref{cor4G}. 
One may refer to this as \emph{Herzog's conjecture}. 
In the domain case, this appears in \cite[Ex.~II.6.9]{weibel}. 
As another application, see Corollary~\ref{logor}.

A module $\mathcal{M}$ is called \emph{weakly algebraic} if it is a direct summand of an algebraic module. 
The following observation extends \cite[Proposition~4.3]{bea} by dropping the torsion-free assumption via a new argument.

\begin{observation}
	Let $R$ be analytically unramified and of dimension one. 
	Then every complete module is weakly algebraic. 
	Conversely, assume $R$ is essentially of finite type over a prime field and every complete $\widehat{R}$-module is weakly algebraic. 
	Then $\dim R \leq 1$.
\end{observation}

In Proposition~\ref{bm} and Corollary~\ref{bm1}, we present a connection from Observation~1.2 to the derived category of modules. 
In certain low-dimensional cases, there are abundant examples of strongly algebraic modules. 
For instance, we show the following.

\begin{observation}
	Let $R$ be a $d$-dimensional local ring and $\mathcal{M}$ a complete module of finite projective dimension. 
	Suppose one of the following holds:
	\begin{itemize}
		\item[(i)] $d=1$;
		\item[(ii)] $d=2$ and $\mathcal{M}$ is torsionless;
		\item[(iii)] $d=3$ and $\mathcal{M}$ is reflexive.
	\end{itemize}
	Then $\mathcal{M}$ is strongly algebraic.
\end{observation}

Despite the restriction on projective dimension, part (i) may be compared with a result of Levy and Odenthal~\cite{levy}, who worked over $1$-dimensional analytically unramified rings. 
Similarly, part (ii) may be compared with a result of Weston~\cite{two}, who worked over $2$-dimensional analytically normal domains. 
In Theorem~\ref{Aus}, we reformulate and unify an algebrization method due to Horrocks, Auslander--Bridger, and Peskine--Szpiro. 
We then present some of its applications, such as the algebrization process for certain generalized Cohen--Macaulay modules.
These results has an application. In particular we extend a   recent result of Puthenpurakal \cite[Theorem 1.1]{tony}:

\begin{corollary} Let $(A,\fm)$ be an excellent Gorenstein isolated singularity of dimension
	$d \geq 2$. Let $f : G(A) \to G(\widehat{A})$ be the natural map. The following are equivalent: \begin{enumerate}
		\item[(i)]   $f_\mathbb{Q}$ is an isomorphism.
		
		\item[(ii)]  For any maximal Cohen-Macaulay (abb. MCM) $\widehat{A}$-module $M$ there exists an MCM $A$-module $N$ and integers
		$r \geq 1$ and $s\geq  0$  such that ${M}^r \oplus\widehat{A}^s
		\cong \widehat{N}.$
		\item[(iii)]   For any locally free over punctured spectrum  $\widehat{A}$-module $M$ there exists an  $A$-module $N$ and integers
		$r \geq 1$ and $s\geq  0$  such that ${M}^r \oplus\widehat{A}^s
		\cong \widehat{N}.$
	\end{enumerate}
\end{corollary}
Similarly, one can extend \cite[Theorem 1.4]{tony} by working with the natural morphism $\theta:\underline{CM}({A})\to \underline{CM}({\widehat{A}})$, instead of $f$,
where $ \underline{CM}({A})$ is the stable category of maximal Cohen-Macaulay $A$-modules.
Section~4 collects several remarks on the algebrization of certain formal regular functions. 
Let $\mathcal{X} := \Spec(R) \setminus \{\fm\}$ and let $(\widehat{\mathcal{X}}, \mathcal{O}_{\widehat{\mathcal{X}}})$ denote the formal completion of $\mathcal{X}$ along $\mathcal{Y} := \V(\fa) \setminus \{\fm\}$. 
Following Hironaka and Matsumura~\cite{hiro}, the subscheme $\mathcal{Y}$ is called $\Gg$ in $\mathcal{X}$ provided
\[
\HH^0(\mathcal{X}, \mathcal{O}_{\mathcal{X}}) \stackrel{\simeq}{\longrightarrow} \HH^0(\widehat{\mathcal{X}}, \mathcal{O}_{\widehat{\mathcal{X}}}).
\]
We present situations for which $\mathcal{Y}$ is (or is not) $\Gg$ in $\mathcal{X}$. 
For example, over a complete Cohen--Macaulay local ring of dimension $d>1$, we show that $\mathcal{Y}$ is $\Gg$ in $\mathcal{X}$ if and only if $\cd(\fa) \leq d-2$. 
Recall from \cite[Corollary~4.2]{AD} that a nonzero module $M$ is Cohen--Macaulay if and only if the following formula is valid
\[
\fgrade(\fb, M) + \cd(\fb, M) = \dim(M) \quad \forall\, \fb \lhd R. \qquad (\ast)
\]
Here, $\fgrade(\fb,-)$ denotes the \emph{formal grade}. 
It was conjectured in \cite[Page~1098]{AD} that the Cohen--Macaulay condition can be replaced by a weaker one under suitable restrictions on $(\ast)$. 
We confirm that prediction by showing the following.

\begin{corollary}
	Let $R$ be a local domain of dimension $d \geq 3$, and let $M$ be a torsion-free module satisfying Serre's $\Se_r$ condition. 
	Suppose $R$ is complete with respect to the $\fa$-adic topology and is a homomorphic image of a Gorenstein ring. 
	If $\cd(\fa, M) \leq \dim M - r$, then $\fgrade(\fa, M) \geq r$.
\end{corollary}

It is worth noting that the special case $r=2$ of the last observation extends \cite[Corollary~2]{fal2}, where Faltings worked under the assumption $\mu(\fa) \leq d-2$. 
In his setting, $R$ is a homomorphic image of a regular ring. 
This, in turn, implies that
\(
\HH^0(\widehat{\mathcal{X}}, \widehat{\mathcal{F}}) \simeq \HH^0(\mathcal{X}, \mathcal{F}),
\)
where $\mathcal{F}$ is the sheaf associated to $M$ over $\mathcal{X}$ and $\widehat{\mathcal{F}}$ is its formal completion along $\mathcal{Y}$. 
Finally, we provide an elementary proof of a remarkable result of Bhatt and de~Jong; see Proposition~\ref{bhde}. Despite its simplicity, we obtain \ref{bhde} independent of them.

\section{ Algebraic modules}

The notation $\Mod(-)$ stands for the category of finitely generated modules.
Free modules are algebraic, and the class of finite-length modules is also algebraic.
More generally, we have the following.

\begin{discussion}\label{art}
	\begin{enumerate}
		\item[(i)] 
		Let $\mathcal{M}$ be a complete module. Suppose that $\mathcal{M}$ is finitely generated as an $R$-module. 
		Then $\mathcal{M}$ is strongly algebraic. Indeed, let $\mathcal{N}$ be a direct summand of $\mathcal{M}$. 
		Note that $\mathcal{N}$ is finitely generated as an $R$-module. 
		By \cite[Theorem~1.8]{sean}, we have $\mathcal{N} \simeq \mathcal{N} \otimes_R \widehat{R}$, and hence $\mathcal{N}$ is algebraic. 
		In particular, any finite-length $\widehat{R}$-module is algebraic.
		
		\item[(ii)] 
		Let $A$ be an Artinian $R$-module. 
		It is well known that $A$ can be equipped with the structure of an $\widehat{R}$-module in such a way that the original $R$-module structure is recovered from the canonical map $R \to \widehat{R}$. 
		In this regard, $A \simeq A \otimes_R \widehat{R}$.
		
		\item[(iii)] 
		Let $(\widehat{R}, \fm_{\widehat{R}})$ be any complete local ring, and let $\fa$ be an ideal primary to its maximal ideal. 
		Then $\fa$ is algebraic. 
		Indeed, consider the short exact sequence
		\(
		0 \longrightarrow \fa \longrightarrow \widehat{R} \longrightarrow \frac{\widehat{R}}{\fa} \longrightarrow 0.
		\)
		Since $\widehat{R}/\fa$ has finite length, it is extended from an $R$-module. 
		Moreover, $\Hom_{\widehat{R}}(\widehat{R}, \widehat{R}/\fa) \simeq \widehat{R}/\fa$ has finite length both as an $\widehat{R}$-module and as an $R$-module. 
		In particular, $\Hom_{\widehat{R}}(\widehat{R}, \widehat{R}/\fa)$ is finitely generated as an $R$-module. 
		By \cite[Proposition~3.2(ii)]{sean}, it follows that $\fa$ is algebraic.
	\end{enumerate}
\end{discussion}

\begin{remark}\label{idealto}
	Suppose every ideal of a complete ring is algebraic. 
	Is every complete module algebraic?
	\begin{enumerate}
		\item[(i)] The answer is negative, even for $2$-dimensional regular rings.
		\item[(ii)] Let $\mathcal{R}$ be a $2$-dimensional complete normal local domain. 
		If every ideal of $\mathcal{R}$ is algebraic, then every torsion-free module over $\mathcal{R}$ is algebraic.
	\end{enumerate}
\end{remark}

\begin{proof}
	(i) To see a counterexample, consider $R := \mathbb{Q}[X,Y]_{(X,Y)}$. 
	It is shown in \cite[Example~3.6]{sean} that there exist non-algebraic modules over $\widehat{R}$. 
	We now check that every ideal of $\widehat{R}$ is algebraic. 
	Since flatness behaves well with respect to intersections, the intersection of any algebraic collection of ideals is algebraic. 
	Using this and primary decomposition, it suffices to treat the case of primary ideals. 
	By Discussion~\ref{art}(iii), we may further assume that the primary ideal has height one. 
	Recall that height-one primary ideals in a UFD are principal. 
	Finally, note that principal ideals over an integral domain are free, and that free modules are algebraic 
	(for a modern argument, see Fact~A in Example~\ref{fail} below).
	
	(ii) Let $\mathcal{M}$ be a torsion-free $\mathcal{R}$-module. 
	Then $\mathcal{M}$ contains a free submodule $\mathcal{F}$ such that $\mathcal{M}/\mathcal{F}$ is an ideal $\mathcal{I}$ of $\mathcal{R}$. 
	Consider the short exact sequence 
	\(
	0 \longrightarrow \mathcal{I} \longrightarrow \mathcal{R} \longrightarrow \mathcal{R}/\mathcal{I} \longrightarrow 0,
	\)
	which gives rise to
	\(
	\Ext^1_{\mathcal{R}}(\mathcal{I}, \mathcal{R}) \simeq \Ext^2_{\mathcal{R}}(\mathcal{R}/\mathcal{I}, \mathcal{R}).
	\)
	Since $\operatorname{gldim}(\mathcal{R}_{\fp}) = 1$ for all $\fp \in \Spec(\mathcal{R}) \setminus \{\fm_{\mathcal{R}}\}$, it follows that 
	$\Supp(\Ext^2_{\mathcal{R}}(\mathcal{R}/\mathcal{I}, \mathcal{R})) \subseteq \{\fm_{\mathcal{R}}\}$. 
	Hence $\Ext^1_{\mathcal{R}}(\mathcal{I}, \mathcal{R})$ is of finite length. 
	Let $R$ be such that $\mathcal{R} = \widehat{R}$. 
	Then there exists a short exact sequence
	\(
	\mathcal{D} : 0 \longrightarrow \mathcal{F} \longrightarrow \mathcal{M} \longrightarrow \mathcal{I} \longrightarrow 0.
	\)
	Thus $\mathcal{D} \in \Ext^1_{\widehat{R}}(\mathcal{I}, \mathcal{F})$. 
	Free modules are algebraic, i.e., there exists a free $R$-module $F$ such that $\widehat{F} \simeq \mathcal{F}$. 
	By assumption, $\mathcal{I} \simeq \widehat{I}$. 
	Since $\Ext^1_{\widehat{R}}(\mathcal{I}, \mathcal{F}) \simeq \Ext^1_{R}(I, F) \otimes_R \widehat{R}$, 
	we have that $\Ext^1_{R}(I, F)$ is of finite length as an $R$-module. 
	By Discussion~\ref{art}, 
	\[
	\mathcal{D} \in \Ext^1_{\widehat{R}}(\mathcal{I}, \mathcal{F}) 
	\simeq \Ext^1_{R}(I, F) \otimes_R \widehat{R} 
	\simeq \Ext^1_{R}(I, F).
	\]
	By the Yoneda interpretation of $\Ext^1$, there exists an exact sequence 
	\( 
	D : 0 \longrightarrow F \longrightarrow M \longrightarrow I \longrightarrow 0
	\)
	such that $\mathcal{D} \simeq D \otimes_R \widehat{R}$. 
	By the $5$-lemma, it follows that $\mathcal{M} \simeq \widehat{M}$.
\end{proof}

\begin{example}\label{fail}
Let $\mathcal{R}$ be a two-dimensional complete normal local domain such that  $\Cl(\mathcal{R})\neq 0$. There is a torsion-free algebraic module $\mathcal{M}$ with  a direct summand $\mathcal{N}$ such that
$\mathcal{N}$ is not algebraic.
\end{example}

\begin{proof}
Any complete  local domain of depth at least two can be realize as a completion of a unique factorization domain. This is in \cite{heit}.
Let $R$ be a $\UFD$ such that its completion is $\mathcal{R}$.
Let $I$ be such that $[I]\in\Cl(\widehat{R})$ is nonzero and put $J$ be such that $[J]=[I]^{-1}$. We look at the $R$-modules
$\mathcal{M}:=I\oplus J$ and $\mathcal{N}:=I$.  Note that $0=[\mathcal{M}]\in\im(\Cl(R)\to\Cl(\widehat{R}))$
 and $0\neq[\mathcal{N}]\notin\im(\Cl(R)\to\Cl(\widehat{R}))$. We bring the following result of Weston (see \cite[Proposition 2.15]{lw}):
  \begin{enumerate}
\item[Fact ] A): Let $L$ be a torsion-free and complete $\widehat{R}$-module. Then $L$ is algebraic
if and only if $[L]\in\im(\Cl(R)\to\Cl(\widehat{R}))$.
\end{enumerate} Since $R$ is $\UFD$, its classical group is zero. Thus, $\im(\Cl(R)\to\Cl(\widehat{R}))=0$.
We are going to apply Fact A) to deduce that  $\mathcal{M}$ is algebraic and that its direct summand $\mathcal{N}$ is not algebraic.
\end{proof}

\begin{discussion}\label{2 of 3}There is a useful principle (see \cite[Proposition 3.1]{sean}):
Let $\mathcal{M}$, and $\mathcal{N}$ be complete modules. If two of $\{\mathcal{M},\mathcal{N},\mathcal{M}\oplus \mathcal{N}\}$ are algebraic, then  so the third.
We call this the principle of 2 of 3.
 \end{discussion}

\begin{proposition}\label{a23}
Let $R$ be a   local domain whose integral closure $\overline{R}$ is a finite
$R$-module  but $\overline{R}$  is not local.   If $R$ is of isolated singularity, then $(\overline{R})^{\widehat{}}$ is  algebraic  but not strongly algebraic.
\end{proposition}

\begin{proof}
 There are only finitely many  prime ideals $\{\fp_1,\ldots,\fp_n\}\subseteq\Spec(\overline{R})$ lying over $\fm$, because $\overline{R}$ is finitely generated over
$R$ as a module. Since $\overline{R}$  is not local, $n>1$.
Recall that $\overline{R}$  is indecomposable as an $R$-module but $(\overline{R})^{\widehat{}}\simeq(\overline{R}_{\fp_1})^{\widehat{}}\oplus\ldots\oplus(\overline{R}_{\fp_n})^{\widehat{}}$ is decomposable as an $\widehat{R}$-module. Suppose on the contrary that $(\overline{R}_{\fp_{i_0}})^{\widehat{}}$ is algebraic for some $i_0$.
By  the principle of 2 of 3 (see Discussion \ref{2 of 3}), $\bigoplus_{i\neq i_0} (\overline{R}_{\fp_i})^{\widehat{}}$ is algebraic. Let $P$ and $Q$ be finitely generated $R$-modules such
that  $\widehat{P}\simeq(\overline{R}_{\fp_{i_0}})^{\widehat{}}$  and $\widehat{Q}\simeq(\bigoplus_{i\neq i_0}\overline{R}_{\fp_{i}})^{\widehat{}}$. Thus, $(Q\oplus P)^{\widehat{}}\simeq (\overline{R})^{\widehat{}}$. Let $\fq\in\Spec(R)\setminus\{\fm\}$. Recall that
$\overline{(R_{\fq})}=(\overline{R})_{\fq}$. Regular rings are normal. Since $R$ is of isolated singularity we deduce that $R_{\fq}=\overline{(R_{\fq})}=(\overline{R})_{\fq}$.
Conclude by this that $\overline{R}$ is locally free on the punctured spectrum.
Since $\overline{R}$ is locally free, $\Ext^1_R(\overline{R},-):\mod(R)\to\mod(R)$ is of finite length. Let $L\in\mod(R)$. Due to Discussion \ref{art}  $\Ext^1_R(\overline{R},L)\simeq\Ext^1_R(\overline{R},L)\otimes_R\widehat{R}$.
 We have
\[\begin{array}{ll}
\Ext^1_R(\overline{R},L)&\simeq\Ext^1_R(\overline{R},L)\otimes_R\widehat{R}\\
&\simeq
\Ext^1_{\widehat{R}}((\overline{R})^{\widehat{}},L\otimes_R\widehat{R}) \\
&\simeq\Ext^1_{\widehat{R}}((Q\oplus P)^{\widehat{}},L\otimes_R\widehat{R})\\
&\simeq\Ext^1_{\widehat{R}}(Q\otimes_R\widehat{R},L\otimes_R\widehat{R})\oplus\Ext^1_{\widehat{R}}(P\otimes_R\widehat{R},L\otimes_R\widehat{R}).
\end{array}\]
From this we see $\Ext^1_{\widehat{R}}(\widehat{Q},\widehat{L})\simeq\Ext^1_{R}(Q,L)\otimes_R\widehat{R}$ is finite length as an $\widehat{R}$-module.
Conclude by this that $\Ext^1_{R}(Q,L)$ is finite length as an $R$-module.
In view of Discussion \ref{art}  $$\Ext^1_{\widehat{R}}(\widehat{Q},\widehat{L})\simeq\Ext^1_{R}(Q,L)\otimes_R\widehat{R}\simeq\Ext^1_{R}(Q,L).$$ Similarly,
$\Ext^1_{\widehat{R}}(\widehat{P},\widehat{L})\simeq\Ext^1_{R}(P,L)\otimes_R\widehat{R}\simeq\Ext^1_{R}(P,L).$
 Combine these together to see
$$\Ext^1_R(\overline{R},L)\simeq\Ext^1_{R}(Q,L)\oplus\Ext^1_{R}(P,L)\quad(+)$$
 Now, we borrow some lines from Hartshorne's coherent functors \cite{har}. Note that $(+)$  implies that $\overline{R}$ and $P\oplus Q$  are stably equivalent  modules. Projective modules are free over local rings. By splitting off free modules, every
stable equivalence class has a unique smallest element containing no free
direct summands. Since $\overline{R}$ is  minimal, we   assume that $P$ and $Q$ are also minimal. Thus $\overline{R}\simeq P\oplus Q$. To see a contradiction recall that $\overline{R}$ is of rank one. Due to this contradiction, we observe that $(\overline{R})^{\widehat{}}$ is   not strongly algebraic.
\end{proof}

Here, we  present two applications of algebrization.
Recall that
Ferrand and Raynaud constructed  1-dimensional  Cohen-Macaulay local rings that are not  homomorphic image of  Gorenstein local  rings.

\begin{corollary}\label{logor}
Let $\widehat{R}$ be any 1-dimensional  Cohen-Macaulay local ring.  Suppose
 $\widehat{R}$ is generically Gorenstein. Then $R$ is a homomorphic image of a Gorenstein local ring.
\end{corollary}

\begin{proof}
 Complete  rings admit a canonical module.
Let $\omega_{\widehat{R}}$  be the canonical module of $\widehat{R}$.
Since
 $\widehat{R}$ is generically Gorenstein,  $\omega_{\widehat{R}}$  has a  rank.     From this, either $\omega_{\widehat{R}}$ is an ideal of height one   or $\omega_{\widehat{R}}=\widehat{R}$. The second possibility implies that $R$ is Gorenstein. Without loss of the generality we may assume that $\omega_{\widehat{R}}$ is an ideal of height one.
 Since the ring is 1-dimensional, $\omega_{\widehat{R}}$ is   primary to the maximal ideal. In view of Discussion \ref{art}(iii)
 $\omega_{\widehat{R}}$ is algebraic. Let $M$ be such that $\widehat{M}=\omega_{\widehat{R}}$.
It is easy to see that $M$ is maximal Cohen-Macaulay. In view of $(\frac{R}{\fm})^{\widehat{}}\simeq\Ext^1_{\widehat{R}}((\frac{R}{\fm})^{\widehat{}},\widehat{M})\simeq\Ext^1_{R}(R/\fm,M)^{\widehat{}}$, we see $\Ext^1_{R}(\frac{R}{\fm},M)\simeq\frac{R}{\fm}$, i.e.,  $M$ is of type $1$.
 Also, $\Ext^{>1}_{R}(R/\fm,M)^{\widehat{}}\simeq\Ext^{>1}_{\widehat{R}}((\frac{R}{\fm})^{\widehat{}},\widehat{M})=0$. So, $\Ext^{>1}_{R}(R/\fm,M)=0$, i.e.,
$M$ is  of finite injective dimension. From this, $R$ admits  a canonical module.
It turns out that $R$ is a homomorphic image of a Gorenstein local ring.
\end{proof}

Recall that $\G_0(R):=\frac{\bigoplus_{M\in \Mod(R)}\mathbb{Z}[M]}{([ A]+[C]-[ B]:0\to A\to B \to C\to  0\textit{ is exact })}.$

\begin{corollary} \label{cor4G}Let
$R$ be a $1$-dimensional local ring with algebraically closed residue field such that $R_{\red}$ is  analytically  reducible. Then $\G_0 (R)\simeq \bigoplus_{\min(R)}\mathbb{Z}$.
\end{corollary}

The following  argument may not be the best one:

\begin{proof}
In the case $R$ is complete and equi-characteristic this is in \cite[Proposition 2.2]{hs} (the proof works in the complete case).
Clearly, the residue field of $R$ coincides  with the residue field of $R_{\red}$.
Since $\G_0 (R)=\G_0 (R_{\red})$ and ${\min(R)}={\min(R}_{\red})$ we may assume that $R$ is reduced.
Clearly, the residue field of $R$ coincides  with the residue field of $\widehat{R}$.
 The assignment  $[M]\mapsto[M \otimes_R\widehat{R}]$ induces a morphism $\varphi:\G_0(R)\rightarrow \G_0(\widehat{R})$.
 For simplicity of the  reader we bring the following fact:
\begin{enumerate}
\item[Fact]  A): Let $A$ be a local ring such that  both of $A$ and $\widehat{A}$ are of isolated singularity. Then  $\varphi:\G_0(A)\rightarrow \G_0(\widehat{A})$ is injective, see \cite[Claim 5.1]{kk}.
\end{enumerate}
Both of $R$ and $\widehat{R}$ are of isolated singularity.
 We apply Fact A) along with
 \cite[Proposition 2.2]{hs} to see  $\varphi:\G_0(R)\hookrightarrow\G_0(\widehat{R})\simeq\bigoplus_{\min(\widehat{R})}\mathbb{Z}.$  In fact \cite[Proposition 2.2]{hs} shows that $[\frac{\widehat{R}}{\widehat{\fm}}]=0$. Since the map is injective, we have $[\frac{R}{\fm}]=0$. Any module has a filtration with prime factors. By prime filtration, $$\G_0 (R)=\langle[R/ \fp]:\fp\in\spec(R)\rangle=\langle[R/ \fp]:\fp\in\min(R)\rangle.$$
To finish the proof, we claim that $\{[R/ \fp]:\fp\in\min(R)\}$ are $\mathbb{Z}$-linearly independent. To this end, and in view of Claim A), we need to show  $\{\varphi([R/ \fp]):\fp\in\min(R)\}$ are $\mathbb{Z}$-linearly independent. Note that $\widehat{R}/ \widehat{\fp} $ has a filtration with prime factors. This prime factors belong to support of $\widehat{R}/ \widehat{\fp} $. On the other hand
$\Supp( \widehat{R}/ \widehat{\fp} )=\Ass_{\widehat{R}}(R/  \fp\otimes_R\widehat{R})\cup \{\widehat{\fm}\}$. Since
$[\frac{\widehat{R}}{\widehat{\fm}}]=0$ we see that $[\varphi(R/ \fp)]\in\langle[\frac{\widehat{R}}{ \mathcal{P}}]:\mathcal{P}\in\Ass(\frac{\widehat{R}}{ \widehat{\fp}})\rangle.$
In view of \cite[Theorem 23.2]{Mat}, $\Ass_{\widehat{R}}(R/  \fp\otimes_R\widehat{R})=\{\mathcal{P}\in\Spec(\widehat{R}):\mathcal{P}\cap R=\fp\}.$ These belongs to $\min(\widehat{R})$.
Note that if $\fp$ are $\fq$ are distinct minimal primes, then $\Ass_{\widehat{R}}(R/  \fp\otimes_R\widehat{R})\cap\Ass_{\widehat{R}}(R/  \fq\otimes_R\widehat{R})=\emptyset.$ We conclude
by this that $\{\varphi([R/ \fp]):\fp\in\min(R)\}$ are $\mathbb{Z}$-linearly independent. The proof is now complete.
\end{proof}

\section{weakly and strongly algebraic modules}

Evans   asked which  local rings satisfy   Krull-Remak-Schmidt
property.

\begin{observation}\label{schmit}
Let $R$ be a local ring that fails  Krull-Remak-Schmidt
property. Then there is a weakly algebraic module  which is not algebraic.
\end{observation}

\begin{proof}
By $X|Y$ we mean $X$ is a direct summand of $Y$.
Due to the assumption  there are  finite family of indecomposable modules $\{M_i\}_{i}$ and $\{N_j\}_{j}$ such that
$\bigoplus_i M_i\simeq \bigoplus_j N_j$ but $M_1$ is not isomorphic to each of $N_j$.  Since completion commutes with finite direct sum, $\bigoplus_i \widehat{M}_i\simeq \bigoplus_j \widehat{N}_j$.  Suppose on the contrary that $\widehat{M}_1$ is indecomposable as an $\widehat{R}$-module.
As complete local rings
satisfy  Krull-Remak-Schmidt (see \cite[Corollay 1.10]{lw}), we should have $\widehat{M}_1\simeq \mathcal{N}$ where $\mathcal{N}|\widehat{N}_j$ for some $j$.
Thus, $\widehat{M}_1| \widehat{N} _j$.
  \begin{enumerate}
\item[Fact ] A): (see the proof of \cite[Proposition 3.1]{sean}) Let $A\to B$ be a flat ring homomorphism, $L_1$ and $L_2$ be two finitely generated $A$-modules. If $L_1\otimes B | L_2\otimes B$, then
$L_1| L_2 $.
\end{enumerate}
In view of Fact A), $M_1| N _j$. This contradiction shows that $\widehat{M}_1$ is  decomposable: $\widehat{M_1}\simeq\mathcal{A}\oplus \mathcal{B}$ for some $\widehat{R}$-modules
$\mathcal{A}$ and $\mathcal{B}$. It is now clear that both of $\mathcal{A}$ and $\mathcal{B}$ are weakly algebraic. Suppose on the contradiction both  of $\mathcal{A}$ and $\mathcal{B}$ are algebraic. There are finitely generated $R$-modules $A$ and $B$ such that $\mathcal{A}\simeq\widehat{A}$ and $\mathcal{B}\simeq\widehat{B}$. This shows that $\widehat{A}|\widehat{M}_1$. Another use of Fact A) shows that $A|M_1$. This contradicts the fact that $M_1$ is indecomposable.
By principle 2 of 3, $\mathcal{A}$ (resp. $\mathcal{B}$) is a weakly algebraic module  and it is not algebraic.
\end{proof}

\begin{corollary} \label{krs}(Levy-Odenthal) Let
$R$ be a $1$-dimensional analytically irreducible local ring. Then $R$ satisfies   Krull-Remak-Schmidt property.
\end{corollary}

\begin{proof}
By the above observation its enough to show any module is algebraic. This is  in \cite{levy} (see \cite[Proposition 3.3]{sean} for a  simple proof).
\end{proof}

We need the following implicit result from \cite{sean}:

\begin{fact}\label{corelkik}
Let $\mathcal{M}$ be   complete and locally free over the punctured spectrum. Then $\mathcal{M}$ is weakly algebraic.\end{fact}\begin{proof}
Indeed, the  map $R \to \widehat{R}$ factors through the henselization $R^h$. By a result of Elkik (see \cite[Theorem 10.10]{lw}) there is a finitely generated $R^h$-module $M_h$ such that $M_h\otimes_{R^h}\widehat{R}\simeq \mathcal{M}$. By \cite[Corollary 3.5]{sean}   there is a
finitely generated  $R$-module $M$ such that $M_h\oplus L_h\simeq M\otimes_RR^h$  for some $R^h$-module $L_h$. So,
 $$\mathcal{M}\oplus(L_h\otimes_{R^h}\widehat{R})\simeq (M_h\otimes_{R^h}\widehat{R})\oplus (L_h\otimes_{R^h}\widehat{R})\simeq( M\otimes_RR^h)\otimes_{R^h}\widehat{R}\simeq M\otimes_R\widehat{R}.$$
\end{proof}

\begin{proposition}\label{bm2} Let $(R,\fm)$ be a local ring. The following assertions are true:
\begin{enumerate}
\item[i)]  Suppose $R$ is analytically unramified and of dimension one.  Then any complete module is weakly algebraic.
\item[ii)] Suppose $R$ is essentially of finite type over a prime field and any complete $\widehat{R}$-module is weakly algebraic. Then $\dim(R)\leq 1$.
\end{enumerate}
\end{proposition}

\begin{proof}i) Let
$\mathcal{M}$ be any complete module.
Since $\widehat{R}$ is reduced, $\widehat{R}$ satisfies  Serre's condition  $\R_0$. This means that  $\widehat{R}$ is  regular over $\mathcal{X}:=\Spec(\widehat{R})\setminus\{\widehat{\fm}\}$. Zero-dimensional regular local rings
are field. In particular, any module over a such ring is a vector space.   We conclude by this that $\mathcal{M}_{\fp}$ is free over $\widehat{R}_{\fp}$ for all $\fp\in \mathcal{X}$. In view of Fact \ref{corelkik},
$\mathcal{M}$ is weakly algebraic.

ii) Recall that the prime field in characteristic zero is $\mathbb{Q}$ and in the characteristic $p$ is $\mathbb{F}_p$.
Since $R$ is essentially of finite type over a prime field, we conclude that $R$ is countable. We are going to use a trick that we learn
from \cite[Example 3.6]{sean}. As $R$ is countable, the
class of finitely generated modules is countable up to the isomorphism class. This implies that the class
of algebraic modules is countable, again up to the isomorphism class. The same thing holds for weakly algebraic modules. Suppose on the contradiction that $\dim R>1$.
It follows by countable prime avoidance for complete rings that there are uncountable family of height one prime ideals $\{\mathcal{P}:\mathcal{P}\in\Gamma\}$.
This implies that $\{\frac{\widehat{R}}{\mathcal{P}}:\mathcal{P}\in \Gamma\}$ is an uncountable family of complete modules up to isomorphism.
So, there is a complete module that is not weakly algebraic, a contradiction. \end{proof}

If $x\in \widehat{R}$ is not zero-divisor then $(x)\simeq\widehat{R}$. So, it is algebraic.  What can say when $x$ is  zero-divisor?
Let us revisit $R_0:=\mathbb{C}[x, y ]_{(x,y )}/(y^2 - x^3 - x^2)$ and recall that $\widehat{R}_0\simeq\mathbb{C}[[u, v ]]/(UV )$.
It is shown in \cite[Example A.5]{kmv} that $(u)$ is not algebraic. They showed that $(u)\oplus(v)$
is algebraic  via three different arguments: 1)  Thomason's localization theorem, 2) Levy-Odenthal's  criterion, and 3) a direct calculation over a triangulated category. This is a sample of:

\begin{proposition}\label{bm}
Let $\widehat{R}$ be any 1-dimensional reduced local ring. Let
 $0\neq x\in \widehat{R}$ be a zero-divisor. Let $\fp\in\Ass(\widehat{R})$ be such that $x\in\fp$.  Then $\widehat{R}/(x)\oplus \widehat{R}/(y)$ is algebraic  for any  $y\in \bigcap_{\fq\in\Ass(R)\setminus\fp} \fq$.
\end{proposition}

\begin{proof}
Since $\widehat{R}$ has a zero divisor $0\notin\Ass(\widehat{R})$.
Reduced rings satisfy Serre's condition $\Se_1$. Hence $\Ass(\widehat{R})=\min(\widehat{R})$.
If $\Ass(\widehat{R})$ were be singleton e.g. $\Ass(\widehat{R})=\{P\}$ for some $P$, then we should had $P\neq0$ (because there are non-trivial zero-divisors) and so there are nilpotent elements.
But the ring is reduced. This shows  that $|\Ass(\widehat{R})|\geq2$.
Let
$\Ass(\widehat{R})=\min(\widehat{R})=:\{\fp,\fq_1,\ldots,\fq_n\}$.
 Due to the primary decomposition,  $(x)\cap(y)\subseteq \fp\cap (\bigcap\fq_i)=\nil(\widehat{R})=0$. Look at the exact sequence
$$0\lo \widehat{R}\lo \widehat{R}/(x)\oplus \widehat{R}/(y)\lo \widehat{R}/(x,y)\lo 0 \quad(\ast)$$
Recall that $\fp\cap (\bigcap\fq_i)=0$. As $x\notin \fq_i$ and $y\notin \fp$, we have $\rad(x,y)=\fm$. Thus,  $\widehat{R}/(x,y)$ is of finite
 length.  Due to Discussion \ref{art} $\widehat{R}/(x,y)$ is algebraic. Suppose it extended from an $R$-module $M$. Since
 $\Ext^1_{\widehat{R}}(\widehat{R}/(x,y),\widehat{R})$ is of finite length as an $\widehat{R}$-module, we have
$\Ext^1_{R}(M,R)\simeq\Ext^1_{R}(M,R)\otimes_R{\widehat{R}}\simeq\Ext^1_{\widehat{R}}(\widehat{R}/(x,y),\widehat{R})
$. We deduce from Yoneda's
definition of $\Ext^1$ that there is an exact sequence $0\to R\to N\to M \to 0$ such that its completion is
   $(\ast)$. By  5-lemma
 $\widehat{N }\simeq\widehat{R}/(x)\oplus \widehat{R}/(y)$.
\end{proof}

The singular category of a ring $A$ is the Verdier quotient
$D_{Sg}(A) := \frac{\D^b(\mod A)}{\K^b(\proj A)}$. By $\Sigma:D_{Sg}(A)\to D_{Sg}(A)$ we mean the suspension functor.
The following was proved by using derived-categorical  methods in \cite[Example A.5]{kmv}:

\begin{corollary}\label{bm1}
 Let $\widehat{R}_0\simeq k[[u, v ]]/(uv )$. Then
 $(u)\oplus\Sigma (u)=k\in(D_{Sg}(\widehat{R}_0),\Sigma)$.
\end{corollary}

\begin{proof}
 Recall that $\frac{\widehat{R}_0}{(u)}\simeq(v)$ and $\frac{\widehat{R}_0}{(v)}\simeq(u)$.
From this $\Sigma (u)=(v)$. In  $D_{Sg}(\widehat{R}_0)$ the class of free modules realized as a zero object. Apply this along with the  exact sequence
$0\to{\widehat{R}}\to\frac{\widehat{R}_0}{(u)}\oplus\frac{\widehat{R}_0}{(v)} \to k \to 0$
to see  $(u)\oplus\Sigma (u)=k\in D_{Sg}(\widehat{R}_0)$.
\end{proof}

Let $(R,\fm)$ be a regular local ring. Horrocks proved that
any complete module  free over
$\Spec(\widehat{R})\setminus\{\fm_{\widehat{R}}\}$  is algebraic, see \cite[\S 8]{h}.
This extends in the following sense:

\begin{theorem}(Auslander-Bridger, Peskine-Szpiro)\label{Aus}
Let $\mathcal{M}$ be a  complete module of finite projective dimension such that
$\Ext^+_{\widehat{R}}(\mathcal{M},\widehat{R})$ is of finite length. Then $\mathcal{M}$ is strongly algebraic.
\end{theorem}

\begin{proof}
It is shown in \cite[Proposition 5.44]{brid}  that  there is an $n\in \mathbb{N}_0$ and a finitely generated $R$-module $N$ such that  $\mathcal{M}\oplus \widehat{R}^n\simeq\widehat{N}$.
Recall that complete free modules are algebraic.
By  principle 2 of 3 (see Discussion \ref{2 of 3}) $\mathcal{M}$ is algebraic. We apply this for any direct summand of $\mathcal{M}$  to conclude that $\mathcal{M}$ is strongly algebraic.
\end{proof}

\begin{corollary} \label{ty} Let $(A,\fm)$ be an excellent Gorenstein isolated singularity of dimension
	$d \geq 2$. Let $f : G(A) \to G(\widehat{A})$ be the natural map. The following are equivalent: \begin{enumerate}
		\item[i)]   $f_\mathbb{Q}$ is an isomorphism.
		
		\item[ii)]  For any MCM  $\widehat{A}$-module $M$ there exists an MCM $A$-module $N$ and integers
		$r \geq 1$ and $s\geq  0$  such that ${M}^r \oplus\widehat{A}^s
	\cong \widehat{N}.$
		\item[iii)]   For any locally free over punctured spectrum  $\widehat{A}$-module $M$ there exists an  $A$-module $N$ and integers
		$r \geq 1$ and $s\geq  0$  such that ${M}^r \oplus\widehat{A}^s
		\cong \widehat{N}.$
	\end{enumerate}
\end{corollary}

\begin{proof}
$ii)\Rightarrow iii)$:
Recall that G-ring part of the definition of excellence, and over 
G-rings, completion of isolated singularity is as well.
Let $M$ be locally free over punctured spectrum  $\widehat{A}$-module.
Now, we use Auslander-Buchweitz approximation (see \cite[Theorem 11.17]{lw}). Namely, there is an exact sequence $$\zeta:0\lo I\lo C\lo M\lo 0,$$ of $\widehat{A}$-modules where $C$ is maximal Cohen-Macaulay, and $I$ is of finite injective dimension.
By assumption, there exists an  $A$-module $N_1$ and integers
$r \geq 1$ and $s\geq  0$  such that ${C}^r \oplus\widehat{A}^s
\cong \widehat{N}_1.$
Since  $\widehat{A}$ is Gorenstein and $\id(I)<\infty$, we deduce that $\pd(I)<\infty$. Since $M$ is locally free,  $\zeta$ locally splits over punctured spectrum. As
 $\widehat{A}$ is of isolated singularity, we conclude from Auslander-Buchsbaum formula that $C$ is
 locally free over punctured spectrum. Combining these together, it turns out that   $I$ is locally free over punctured spectrum. In particular, we deduce that $\Ext^+_{\widehat{A}}(I^r,\widehat{A})$ is of finite length, which enables us to apply Theorem \ref{Aus}. By previous theorem, there is a finitely generated $A$-module $N_2$ and $t>0$ so that $I^r \oplus\widehat{A}^t
 \cong \widehat{N}_2.$ 
 We then look at the exact sequence$$\zeta_1:0\lo I^r\oplus \widehat{A}^{r+t}\lo C^r\oplus \widehat{A}^{r+t}\lo M^r\oplus \widehat{A}^{r+t}\lo 0,$$
By using Discussion \ref{2 of 3} we get that  $M^r\oplus \widehat{A}^{r+t}$ is algebraic claim.
 
 $iii)\Rightarrow  ii)$ Easy.
 
 The other implications are in \cite[theorem 1.1]{tony}.
\end{proof}\begin{remark}
Similarly, one can extend \cite[Theorem 1.4]{tony} by replacing the map $f$ from Corollary \ref{ty}(i) with the natural map $\theta:\underline{CM}({A})\to \underline{CM}({\widehat{A}})$,  where $ \underline{CM}({A})$ is the stable category of maximal Cohen-Macaulay $A$-modules which is a triangulated category.  
\end{remark}
Let us revisit again $R_0:=\frac{\mathbb{C}[x, y ]_{(x,y )}}{(y^2 - x^3 - x^2)}$. Then $\widehat{R}_0\simeq\mathbb{C}[[u, v ]]/(uv )$. Since $\mathcal{M}:=\widehat{R}_0/(u)$ is locally free on the punctured spectrum, we have $\ell(\Ext^{+}_{\widehat{R}_0}(\mathcal{M},\widehat{R}_0))<\infty$. But, $\mathcal{M}$ is not algebraic (see \cite[Page 335]{sean}).
Thus, finiteness of projective dimension in Theorem \ref{Aus} is important (also, we can't replace finiteness of projective dimension by finiteness of G-dimension).

\begin{discussion}\label{sph}
Let $M$ and $N$ be finitely generated modules such that either $\pd(M)<\infty$
or $\id(N)<\infty$ over any commutative noetherian ring. Let $i\leq d:=\dim R$.
Then $\dim(\Ext^{i}_R(M,N))\leq d-i$. The origin source of this is \cite{sga2}
by Grothendieck  (claim in the case  $\id(N)<\infty$ follows easily, see e.g. \cite{a}).
\end{discussion}

\begin{corollary}\label{sys}
Let $\widehat{R}$ be a $d$-dimensional local ring and $\mathcal{M}$
a complete module of finite projective dimension which is $(d-1)$-syzygy. Then  $\mathcal{M}$ is strongly algebraic.
\end{corollary}

\begin{proof}
We assume that $\mathcal{M}\neq 0$ and that $d>0$. First, suppose that $d>1$. By definition of syzygy modules, there is  a complete module
$\mathcal{N}$ and the  exact sequence
$
0 \to\mathcal{M}\to\widehat{R}^{\beta_{d-2}} \to\ldots\lo \widehat{R}^{\beta_0} \to \mathcal{N}\to 0$.
We have $\pd(\mathcal{M})\leq1$.  In the case $d=1$, we apply Auslander-Buchsbaum formula  to see $\pd(\mathcal{M})\leq \depth(\widehat{R})\leq\dim(\widehat{R})=1$. In both cases  $\pd(\mathcal{M})\leq1$. Projective modules are free over local rings, and free modules are strongly algebraic.   Without loss of generality we assume that $\pd(\mathcal{M})=1$.
 We have $\Ext^1_{\widehat{R}}(\mathcal{M},\widehat{R})\simeq\Ext^d_{\widehat{R}}(\mathcal{N},\widehat{R})$.
By Discussion \ref{sph}, $$\dim(\Ext^1_{\widehat{R}}(\mathcal{M},\widehat{R}))=\dim(\Ext^d_{\widehat{R}}(\mathcal{N},\widehat{R}))\leq \dim \widehat{R}-d=0.$$
Consequently, $\Ext^+_{\widehat{R}}(\mathcal{M},\widehat{R})$ is of finite length.  By   Theorem \ref{Aus} $\mathcal{M}$ is strongly algebraic.
\end{proof}

\begin{corollary}\label{ref3}Let
$R$ be  $d$-dimensional and $\mathcal{M}$ a complete module of finite projective dimension. Suppose one of the following holds:
\begin{enumerate}
\item[i)]   $d=1$,
\item[ii)]  $d=2$ and $\mathcal{M}$ is torsion-less, or
\item[iii)]   If $d=3$ and $\mathcal{M}$ is reflexive.
\end{enumerate}
Then $\mathcal{M}$ is strongly algebraic.
\end{corollary}

One may replace ii) with the following statement: If $d=2$, $\mathcal{M}$ is torsion-free and $\widehat{R}$ satisfies Serre's condition  $\R_1$ (resp. $\widehat{R}$ is an integral domain),
then $\mathcal{M}$ is strongly algebraic.

\begin{proof}
i)    We may assume that $\pd(\mathcal{M})=1$. In the light of Discussion \ref{sph}, we observe that
 $\dim(\Ext^1_{\widehat{R}}(\mathcal{M},\widehat{R}))=0$. Therefore, $\Ext^+_{\widehat{R}}(\mathcal{M},\widehat{R})$ is of finite length.
 Theorem \ref{Aus} yields the claim.

ii) Let $\widehat{R}^n\to \mathcal{M}^{\ast}\to 0$ be a presentation. Applying $\Hom_{\widehat{R}}(-,\widehat{R})$ to the presentation implies that $ 0\to \mathcal{M}^{\ast\ast}\to \widehat{R}^n$ is exact. By torsion-less, $\mathcal{M}\subset \mathcal{M}^{\ast\ast}$. Combine these to see
$\mathcal{M}$ is 1-syzygy.  It remains to use Corollary \ref{sys}.

iii) Let $\widehat{R}^{n}\to \widehat{R}^m\to \mathcal{M}^{\ast}\to 0$ be a finite presentation. Then $0\to\mathcal{ M}^{\ast\ast}\to \widehat{R}^m\to \widehat{R}^n$ is exact. Since $\mathcal{M}$ is reflexive
it follows that it is  second syzygy. Due to  Corollary \ref{sys} $\mathcal{M}$ is strongly algebraic.
\end{proof}

Let $R$ be a ring, $\mathfrak a$  an ideal with a generating set
 $\underline{a}:=a_{1}, \ldots, a_{r}$.
 By $\HH_{\underline{a}}^{i}(M)$, we mean the $i$-th cohomology of the
\textit{$\check{C}ech$} complex of  a module $M$ with respect  to $\underline{a}$. This is independent of the choose
of the generating set. For simplicity, we denote it by $\HH_{\mathfrak a}^{i}(M)$.

\begin{corollary}\label{gcm}
Let  $\widehat{R}$ be
an analytically irreducible  local ring which  is Gorenstein over its punctured  spectrum. Let  $\mathcal{M}$ be
a generalized Cohen-Macaulay module and of finite projective dimension. If $\dim(\mathcal{M})=\dim \widehat{R}$, then  $\mathcal{M}$ is strongly algebraic.
In fact, $\mathcal{M}$ extended from a generalized Cohen-Macaulay $R$-module.
\end{corollary}

\begin{proof}
We use a trick taken from \cite{gcm}.
Set $\f_{\fa}(\mathcal{N}):=\inf\{i:\HH^{i}_{\fa}(\mathcal{N}) \text{ is not finitely
generated}\}$.   Recall that $\widehat{R}$ is a quotient of a regular  ring.  By Grothendieck's finiteness theorem  $$\f_{\fa}(\mathcal{M})=\inf\{\depth(\mathcal{M}_{\fq})+\Ht(\frac{\fa+\fq}{\fq}):\fq\in\Supp(\mathcal{M})\setminus\V(\fa)\}.$$Then, $\dim \widehat{R}=\f_{\fm_{\widehat{R}}}(\mathcal{M}).$
Since $\dim(\mathcal{M})=\dim \widehat{R}$ and   $\widehat{R}$ is domain, we see that $\Spec(\widehat{R})=\Supp(\mathcal{M})$.
Let $\fq\in\Spec(\widehat{R})\setminus\V(\fm_{\widehat{R}})$. Then
 $\fq\in\Supp(\mathcal{M})\setminus\V(\fm_{\widehat{R}})$ and so $$\dim \widehat{R}\leq\depth(\mathcal{M}_{\fq}) +\dim(\frac{\widehat{R}}{\fq})\leq\dim(\mathcal{M}_{\fq}) +\Ht(\frac{\widehat{R}}{\fq})\leq\Ht(\fq) +\Ht(\frac{\widehat{R}}{\fq})\leq\dim \widehat{R}.$$
We conclude that $\depth(\mathcal{M}_{\fq})=\dim(\widehat{R}_{\fq})$. We combine this and the
  Gorenstein property of $\widehat{R}_{\fq}$ along with the local duality to see $\Ext^{i}_{\widehat{R}}(\mathcal{M},\widehat{R})_{\fq}\simeq\Ext^{i}_{\widehat{R}_{\fq}}(\mathcal{M}_{\fq},\widehat{R}_{\fq})\simeq\HH^{\dim\widehat{R}_{\fq}-i}_{\fq\widehat{R}_{\fq}}(\mathcal{M}_{\fq})^v=0$ for all $i>0$.
From this we conclude that $\Ext^{+}_{\widehat{R}}(\mathcal{M},\widehat{R})$ is of finite length.
We deduce from   Theorem \ref{Aus} that $\mathcal{M}$ is strongly algebraic.
Let $M$ be such that $\widehat{M}=\mathcal{M}$. To show $M$ is generalized Cohen-Macaulay, its enough to note
that $\HH^{i}_{\widehat{\fm}}(\mathcal{M})\simeq\HH^{i}_{\fm}(M)\otimes_R\widehat{R}\simeq\HH^{i}_{\fm}(M)$ for all $i<\dim M$.
\end{proof}

\begin{remark} In Corollary \ref{gcm}, the assumption $\dim(\mathcal{M})=\dim \widehat{R}$ is important even over  regular rings:   We look at $R:=\mathbb{Q}[X,Y]_{(X,Y)}$.
By Proposition \ref{bm2}   there is a height-one prime ideal $\fp$ such that $\mathcal{M}:=\widehat{R}/ \fp$ is not algebraic. Since $\fp$
is prime, $\HH^0_{(X,Y)}(\mathcal{M})=0$  and so $\mathcal{M}$ is (generalized) Cohen-Macaulay. Note that $\pd(\mathcal{M})=\dim(\mathcal{M})=1<2=\dim \widehat{R}$.
\end{remark}

Recall that an $R$-module $M$  of  projective dimension  $r$ is
called   $r$-spherical if $\Ext^i_{R}(M,R)=0$ for all $i\neq 0,r$. Now, let $\mathcal{M}$ be an $r$-spherical
complete module such that  $\Ext^r_{\widehat{R}}(\mathcal{M},\widehat{R})$ is algebraic. In view of
\cite[Proposition I. 5.8]{ps} $\mathcal{M}$ is strongly algebraic. Now,
let $\mathcal{M}$ be a  complete module of finite projective dimension such that
$\Ext^i_{\widehat{R}}(\mathcal{M},\widehat{R})$ is algebraic for all $i>0$. When is $\mathcal{M}$  algebraic?

Assume that $R_0=k$ is a field of characteristic zero. Let $R = \bigoplus_{i\geq0} R_i$
be a graded k-algebra.  Here, by completion we mean completion with respect to the graded maximal ideal. Let $\mathcal{M}\in\mod(\widehat{R})$ be such that
$\Ext^1_{\widehat{R}}(\mathcal{M}, \mathcal{M}) = 0$. In view of \cite[Proposition 6.1]{kmv} $\mathcal{M}$ is  completion of a finitely generated graded $R$-module
$M$. The graded structure on $R$ is important. To find an example, again we revisit  $R_0:=\mathbb{C}[x, y ]/(y^2 - x^3 - x^2)$.
Recall that $\widehat{R}_0:=\mathbb{C}[[u,v]]/(uv)$ and  $\Ext^1_{\widehat{R}_0}(\widehat{R}_0/(u), \widehat{R}_0/(u)) = 0$.
We ask:\begin{question}
When algebrization depends on the characteristic?
\end{question}

\section{Algerization of formal regular functions}

In this section  $(\mathcal{R},\fm)$ is a complete local ring, and $\fa\lhd \mathcal{R} $. In fact, we only need $\mathcal{R}$ is  complete with respect to $\fa$-adic topology. Let $\mathcal{X}:=\Spec(\mathcal{R})\setminus \{\fm \}$ and $(\widehat{\mathcal{X}},\mathcal{O}_{\widehat{\mathcal{X}}})$ denote the formal
completion of $\mathcal{X}$ along  with $\mathcal{Y}:=\V(\fa)\setminus\{\fm\}$. Recall that
$\mathcal{Y}$ is called $\Gg$ in $\mathcal{X}$ if  $\HH^0(\mathcal{X},\mathcal{O}_{\mathcal{X}})\stackrel{\simeq}\lo \HH^0(\widehat{\mathcal{X}}, \mathcal{O}_{\widehat{\mathcal{X}}})$.
This means that formal regular functions are the ordinary regular functions.

\begin{lemma}
 \label{41} Adopt the above notation.
The following holds:
\begin{enumerate}
\item[i)] If $\Ht(\fa)=\dim \mathcal{R}-1$, then $\HH^0(\widehat{\mathcal{X}},\mathcal{O}_{\widehat{\mathcal{X}}})\simeq(\mathcal{R}_x)^{\wedge_{\fa}}$  for some $x\in\fm$.
\item[ii)]  Assume in addition to  i) that  $\mathcal{R}$ is normal. Then $\HH^0(\widehat{\mathcal{X}},\mathcal{O}_{\widehat{\mathcal{X}}})$  is not finitely generated over $\mathcal{R}$.
\item[iii)] The module $\HH^0(\mathcal{X},\mathcal{O}_\mathcal{X})$ is algebraic over $\mathcal{R}$ provided it is finitely generated over $\mathcal{R}$.
\end{enumerate}
\end{lemma}

\begin{proof}
i) Under some extra assumptions, the desired claim is in \cite{call}.
Indeed, by $\DD_I(-)$ we mean the ideal transformation, and recall
that if $I=(a)$ is principal, it is just the localization. Now,
let $x+\fp$  be a parameter element for the 1-dimensional ring $\frac{\mathcal{R}}{\fa}$, so  that $\rad(\fa+x\mathcal{R})=\fm$.
By the theorem on formal functions, $$\HH^0(\widehat{\mathcal{X}},\mathcal{O}_{\widehat{\mathcal{X}}})\simeq
{\vpl}_n \HH^0(\mathcal{X},\widetilde{\mathcal{R}/ \fa^n})=:{\vpl}_n \DD_{\frac{(x)+\fa^n}{\fa^n}}(\frac{\mathcal{R}}{\fa^n})\simeq{\vpl}_n (\frac{\mathcal{R}_x}{\fa^n\mathcal{R}_x}) =(\mathcal{R}_x)^{\wedge_{\fa}}.$$

ii)
Suppose  on the contrary that $\HH^0(\widehat{\mathcal{X}},\mathcal{O}_{\widehat{\mathcal{X}}})$  is  finitely generated.
Then $(\mathcal{R}_x)^{\wedge_{\fa}}$  is  finitely generated. Hence, its submodule $\mathcal{R}_x$ is as well.
Thus, $1/x$ is integral over $\mathcal{R}$. Since $\mathcal{R}$ is normal, $1/x\in\mathcal{R}$. This contradicts the fact that
$x\in\fm.$

iii) Let $R$ be such that $\widehat{R}\simeq\mathcal{R}$.
There is an  exact sequence $$0\lo\Gamma_{\fm}(\mathcal{R})
\lo\mathcal{R}\lo \HH^0(\mathcal{X},\mathcal{O}_\mathcal{X})\lo \HH^1_{\fm}(\mathcal{R})\lo 0\quad(\ast)$$
Recall that $\Gamma_{\fm}(\mathcal{R})$ is algebraic. Clearly, $\mathcal{R}$ is algebraic.  Also, $\Hom_{\mathcal{R}}(\Gamma_{\fm}(\mathcal{R}),\mathcal{R})$ is of finite length. In particular, $\Hom_{\mathcal{R}}(\Gamma_{\fm}(\mathcal{R}),\mathcal{R})$ is finitely
generated as an $R$-module. Set $\overline{\mathcal{R}}:=\frac{\mathcal{R}}{\Gamma_{\fm}(\mathcal{R})}$ and look at $0\to\Gamma_{\fm}(\mathcal{R})
\to\mathcal{R}\to \overline{\mathcal{R}}\to 0 $.
In view of \cite[Proposition 3.2(ii)]{sean} we see $\overline{\mathcal{R}}$ is algebraic. Suppose it extends
from $M$.
Since $\HH^0(\mathcal{X},\mathcal{O}_\mathcal{X})$ is finitely generated, we see $\HH^1_{\fm}(\mathcal{R})$ is finitely generated.
From this we deduce that $\HH^1_{\fm}(\mathcal{R})$ is of finite length. Due to Discussion \ref{art}
 $\HH^1_{\fm}(\mathcal{R})$ is algebraic. Suppose it extended from $H$.
Now, we look at $\mathcal{D}:0\to\overline{\mathcal{R}}
\to\HH^0(\mathcal{X},\mathcal{O}_\mathcal{X})\to \HH^1_{\fm}(\mathcal{R})\to 0 .$
Then $\mathcal{D}\in \Ext^1_{\mathcal{R}}(\widehat{H},\widehat{M})\simeq
\Ext^1_{R}(H,M)\otimes_R\widehat{R}\simeq\Ext^1_{R}(H,M)$ corresponds
to completion of $0\to M\to D\to H\to 0 $ where $D$ is a finitely generated $R$-module. By 5-lemma,
 $\widehat{D}\simeq \HH^0(\mathcal{X},\mathcal{O}_\mathcal{X}) $. So,
 $\HH^0(\mathcal{X},\mathcal{O}_\mathcal{X})$ is algebraic.
\end{proof}

\begin{corollary}\label{ng1}Let $(\mathcal{R},\fm)$ be a complete normal  local domain of dimension $d\geq2$ and let $\fa$ be  of height  $d-1$. Then  $\mathcal{Y}$ is not $\Gg$ in $\mathcal{X}$.
\end{corollary}

\begin{proof}
Normal rings are $\Se_2$. Apply this along with $d\geq2$ to see
 $\Gamma_{\fm}(\mathcal{R})=\HH^1_{\fm}(\mathcal{R})=0$. In view of $(\ast)$ in Lemma \ref{41}(iii) we see $\HH^0(\mathcal{X},\mathcal{O}_{\mathcal{X}})=\mathcal{R}$.  Due to
Lemma \ref{41}(ii) $\HH^0(\widehat{\mathcal{X}}, \mathcal{O}_{\widehat{\mathcal{X}}})$   is not finitely generated over $\mathcal{R}$.
This implies that $\HH^0(\widehat{\mathcal{X}}, \mathcal{O}_{\widehat{\mathcal{X}}})\ncong\HH^0(\mathcal{X},\mathcal{O}_{\mathcal{X}})$. By definition, $\mathcal{Y}$ is not $\Gg$ in $\mathcal{X}$.
\end{proof}

\begin{discussion}Let $\mathcal{N}$ be a module.
\begin{enumerate}
\item[i)] The cohomological dimension of $\mathcal{N}$
with respect to $\fa$  $\cd(\fa,  \mathcal{N})$ is  the supremum of $i$'s such that $\HH^i_{\fa}(\mathcal{N})\neq 0$.  We use $\cd(\fa)$ instead of $\cd(\fa,  \mathcal{R})$.
\item[i)]Following Peskine and Szpiro,   formal grade  of $\mathcal{N}$ with respect to $\fa$ is the least integer $i$ such that ${\vpl}_n\HH^i_{\fm}(\mathcal{N}/\fa^n \mathcal{N})\neq 0$ and denoted  by $\fgrade(\fa,\mathcal{N})$.\end{enumerate}\end{discussion}

\begin{observation}\label{connec} Let $(\mathcal{R},\fm)$ be a  Cohen-Macaulay local  ring of dimension $d>1$ and complete  with respect to $\fa$-adic topology.   Then $\mathcal{Y}$ is $\Gg$ in $\mathcal{X}$ if and only if $\cd(\fa)\leq d-2$.
\end{observation}

\begin{proof}
Recall that $\depth(\mathcal{R})>1$. The depth condition implies that
$\HH^0(\mathcal{X},\mathcal{O}_{\mathcal{X}})\simeq \mathcal{R}$.
 We look at the following exact sequence (see \cite[\S III.2]{ps})$$
\begin{CD}
0 \to{\vpl}_n\Gamma_{\fm}(\frac{\mathcal{R}}{ \fa^n}) @>>>{\vpl}_n\frac{\mathcal{R}}{ \fa^n} @>>>{\vpl}_n\HH^0(\mathcal{X},\widetilde{\frac{\mathcal{R}}{ \fa^n}}) @>>>{\vpl}_n\HH^1_{\fm}(\frac{\mathcal{R}}{ \fa^n})\to0.
\end{CD}
$$Recall that ${\vpl}_n (\frac{\mathcal{R}}{ \fa^n})=\mathcal{R}$.
 Recall from  \cite[Corollary 4.2]{AD} that the Cohen-Macaulay property  of a nonzero module $\mathcal{N}$ is equivalent with the following property $$\fgrade(\fb,\mathcal{N})+\cd(\fb,\mathcal{N})=\dim(\mathcal{N})\quad \forall \fb\lhd\mathcal{R}\quad(\ast)$$
 Recall that $\Gg$ means that
the natural map $\varphi:\HH^0(\mathcal{X},\mathcal{O}_{\mathcal{X}})\lo \HH^0(\widehat{\mathcal{X}}, \mathcal{O}_{\widehat{\mathcal{X}}})$ is an isomorphism.
Note that $\ker(\varphi)={\vpl}_n\Gamma_{\fm}(\frac{\mathcal{R}}{ \fa^n})$ and $\coker(\varphi)={\vpl}_n\HH_{\fm}^1(\frac{\mathcal{R}}{ \fa^n})$. Thus  $\mathcal{Y}$ is  $\Gg$ in $\mathcal{X}$ if and only if $\fgrade(\fa,\mathcal{R})\geq2$.
In view of $(\ast)$ this holds if and only if
$\cd(\fa)\leq d-2$, as claimed.
\end{proof}

Let us extend the previous observation:

\begin{proposition}
Let $\mathcal{R}$ be a   local domain of dimension $d\geq3$ and $\mathcal{M}$ is torsion-free satisfies Serre's condition $\Se_r$.
Suppose  $\mathcal{R}$ is complete with respect to $\fa$-adic topology and   $\mathcal{R}$ is  homomorphic image of a Gorenstein ring.
If $\cd(\fa,\mathcal{M})\leq \dim\mathcal{M} -r$, then $\fgrade(\fa,\mathcal{M})\geq r$.
\end{proposition}

\begin{proof}
Let $A$ be a ring  $I,\fb\lhd A$ and  $s,\ell\in \mathbb{N}$.
We say   $\left(A, I, T:=\V(\mathfrak b), M, s, \ell\right)$ satisfies the  property $\mathcal{P}$ if we are in the situation of \cite[Tag 0EFU]{Stack}.
Here, we collect the basic properties of $\mathcal{P}$:

\begin{enumerate}
\item[Fact] A): Let $I \subset \mathfrak b$ be ideals of a ring $A$.
Let $M$ be a finite $A$-module. Let $s$ and $\ell$ be integers.
Also, we assume: i)   $A$ has a dualizing complex,
ii)  $\cd(A, I) \leq \ell$,
iii) if $\mathfrak p \not \in \V(I)$ and
$\mathfrak q \in \V(\mathfrak p) \cap \V(\mathfrak b)$ then
$\depth_{A_\mathfrak p}(M_\mathfrak p) > s$ or
$\depth_{A_\mathfrak p}(M_\mathfrak p) +
\dim((A/\mathfrak p)_\mathfrak q) > \ell + s$.
Then $\left(A, I, \V(\mathfrak b), M, s, \ell\right)$ has   property $\mathcal{P}$ (see  \cite[Tag 0EFW]{Stack}).
\end{enumerate}

We look at  $A:=\mathcal{R}, I:=\fa, \V(\mathfrak b)=\fm, M:=\mathcal{R}, s:=r-1,$ and $ \ell:=d-2$. We   are going to show that the  property $\mathcal{P}$ holds for them.
Since $\mathcal{R}$ is  homomorphic image of a Gorenstein ring i)  holds. Since $\mathcal{M}$ is torsion-free, $0\in\Supp(\mathcal{M})$. Hence $\Supp(\mathcal{M})=\Spec(\mathcal{R})$.
Also, cohomological dimension is support sensitive, see \cite{DNT}. Part ii) follows by these.
To show iii), let $\mathfrak p \not \in \V(\fa)$ and
$\mathfrak q \in \V(\mathfrak p) \cap \V(\mathfrak \fm)=\{\fm\}$.
Suppose $\depth (\mathcal{M}_\mathfrak p) > r-1$ is not the case. That is $\depth (\mathcal{M}_\mathfrak p) \leq r-1$.
In the light of  $\depth(\mathcal{M}_\mathfrak p)\geq\min\{r,\dim(\mathcal{M}_\mathfrak p)\}$ we see $\dim(\mathcal{M}_\mathfrak p)=\depth(\mathcal{M}_\mathfrak p)$.
We may assume that $\fp\neq 0$.
Clearly,  $\mathcal{M}_{\fp}$  is torsion-free over $\mathcal{R}_{\fp}$.  Thus, $\Supp_{\mathcal{R}_{\fp}}(\mathcal{M}_{\fp})=\Spec(\mathcal{R}_{\fp})$.
Therefore, $\Ht(\fp)=\dim(\mathcal{M}_{\fp})=\depth(\mathcal{M}_\mathfrak p)$ and consequently
$$\depth (\mathcal{M}_\mathfrak p) +
\dim((\mathcal{R}/\mathfrak p)_\mathfrak q)=\Ht(\fp)+\dim(\mathcal{R}/\mathfrak p)\stackrel{(\natural)}=d > d-1=(d-2)+1=:\ell+ s.$$
Regarding $(\natural)$: we use the fact that any homomorphic image of a Gorenstein ring is  catenary. This clarifies iii).

\begin{enumerate}
\item[Fact]  B): Assume $\mathcal{P}$ holds over a local ring $(A,\fm)$ and $T = \{\mathfrak m\}$. Then
$\HH^i_\mathfrak m(M) \lo {\vpl}_n \HH^i_\mathfrak m(M/I^nM)$
is an isomorphism for $i \leq s$ (see \cite[Tag 0EFR]{Stack}).
\end{enumerate}

We apply Fact   B) to see that  ${\vpl}_n\HH^{<r}_{\fm }(\frac{\mathcal{M}}{ \fa^n\mathcal{M}})\simeq\HH^{<r}_{\fm }({\vpl}_n\frac{\mathcal{M}}{ \fa^n\mathcal{M}})\simeq\HH^{<r}_{\fm }(\mathcal{M})\stackrel{(\ast)}=0$ where $(\ast)$ follows from $\Se_r$.
By definition, $\fgrade(\fa,\mathcal{M})\geq r$.
\end{proof}

Let $\mathcal{F}$  be a sheaf associated to $\mathcal{M}$  over $\mathcal{X}$ and   $\widehat{\mathcal{F}}$ be  its formal completion along with  $\mathcal{Y}$.

\begin{corollary}(After Faltings)
Let $(\mathcal{R},\fm)$ be a   local domain of dimension $d\geq3$,  $\mathcal{M}$ be torsion-free and  $\Se_2$.
Suppose the following holds:
\begin{enumerate}
\item[i)]  $\mathcal{M}$ is complete with respect to $\fa$-adic topology,
\item[ii)]  $\mathcal{R}$ is  homomorphic image of a Gorenstein ring, and
\item[iii)] $\cd(\fa,\mathcal{M})\leq d-2$.
\end{enumerate}
 Then $\HH^0(\widehat{\mathcal{X}},\widehat{\mathcal{F}})\simeq\HH^0(\mathcal{X},\mathcal{F}).$
\end{corollary}

\begin{proof}By the above proposition $\fgrade(\fa,\mathcal{M})\geq 2$. We conclude from$$
\begin{CD}
0 ={\vpl}_n\Gamma_{\fm}(\frac{\mathcal{M}}{ \fa^n\mathcal{M}}) @>>>{\vpl}_n\frac{\mathcal{M}}{ \fa^n\mathcal{M}} @>>>{\vpl}_n\HH^0(\mathcal{X},\frac{\mathcal{F}}{ \fa^n\mathcal{F}}) @>>>{\vpl}_n\HH^1_{\fm }(\frac{\mathcal{M}}{ \fa^n\mathcal{M}})=0\\
  \simeq@AAA\simeq @  AAA @AAA\simeq @AAA   \\
0 =\Gamma_{\fm }(\mathcal{M}) @>>>\mathcal{M} @>>>\HH^0(\mathcal{X},\mathcal{F}) @>>>\HH^1_{\fm }(\mathcal{M})=0\\
\end{CD}
$$and the theorem of formal functions to deduce that $\HH^0(\widehat{\mathcal{X}},\widehat{\mathcal{F}})\simeq{\vpl}_n\HH^0(\mathcal{X},\frac{\mathcal{F}}{ \fa^n\mathcal{F}})\simeq\HH^0(\mathcal{X},\mathcal{F}).$
\end{proof}

We  found that the following is in  \cite{bhat}. The following proof is elementary:

 \begin{proposition}\label{bhde} Let $(\mathcal{R},\fm)$ be a  local  ring, $\fa:=(x)$ and
 $\mathcal{M}$  a finitely generated module with the following properties:
\begin{enumerate}
\item[i)]  $x$ is $\mathcal{M}$-regular,
\item[ii)]   $\mathcal{M}$ is complete with respect to $\fa$-adic topology, and
\item[iii)] $\HH^1_{\fm}(\mathcal{M})$ and $\HH^2_{\fm}(\mathcal{M})$ are annihilated by some power of $x$.
\end{enumerate}
 Then $\HH^0(\widehat{\mathcal{X}},\widehat{\mathcal{F}})\simeq\HH^0(\mathcal{X},\mathcal{F}).$
\end{proposition}

\begin{proof}After changing $x$ with its high powers we may assume that  $x\HH^1_{\fm}(M)=x\HH^2_{\fm}(M)=0$. Let $\theta_n:\frac{\mathcal{M}}{x^{n+1}\mathcal{M}}\to \frac{\mathcal{M}}{x^n\mathcal{M}}$ be the natural surjection. The following diagram:
$$
\begin{CD}
0  @>>>\mathcal{M}@>x^n>>\mathcal{M} @>>>\frac{\mathcal{M}}{x^n\mathcal{M}} @>>>0\\
 @. x@  AAA =@AAA \theta_n@AAA   \\
0@>>>\mathcal{M}@>x^{n+1}>>\mathcal{M} @>>>\frac{\mathcal{M}}{x^{n+1}\mathcal{M}} @>>>0\\
\end{CD}$$
 induces$$
\begin{CD}
\HH^1_{\fm}(\mathcal{M})  @>x^n>>\HH^1_{\fm}(\mathcal{M})@>>>\HH^1_{\fm}(\frac{\mathcal{M}}{x^n\mathcal{M}}) @>>>\HH^2_{\fm}(\mathcal{M}) @>x^n>>\HH^2_{\fm}(\mathcal{M})\\
 @. =@  AAA \pi_n@AAA x@AAA   \\
\HH^1_{\fm}(\mathcal{M})@>x^{n+1}>>\HH^1_{\fm}(\mathcal{M})@>>>\HH^1_{\fm}(\frac{\mathcal{M}}{x^{n+1}\mathcal{M}}) @>>>\HH^2_{\fm}(\mathcal{M}) @>x^{n+1}>>\HH^2_{\fm}(\mathcal{M}),\\
\end{CD}$$where $\pi_n:=\HH^1_{\fm}(\theta_n)$.
Since $x\HH^1_{\fm}(\mathcal{M})=x\HH^2_{\fm}(\mathcal{M})=0$ we have $$\begin{CD}
0  @>>>\HH^1_{\fm}(\mathcal{M})@>>>\HH^1_{\fm}(\frac{\mathcal{M}}{x^n\mathcal{M}}) @>>>\HH^2_{\fm}(\mathcal{M}) @>>>0\\
 @. =@  AAA \pi_n@AAA x@AAA   \\
0@>>>\HH^1_{\fm}(\mathcal{M})@>>>\HH^1_{\fm}(\frac{\mathcal{M}}{x^{n+1}\mathcal{M}}) @>>>\HH^2_{\fm}(\mathcal{M}) @>>>0.\\
\end{CD}$$Note that $\HH^2_{\fm}(\mathcal{M})\stackrel{x}\lo\HH^2_{\fm}(\mathcal{M})$ is a zero map. By Mittag-Leffler, we have $$
\begin{CD}
0  @>>>{\vpl}_n\left(\HH^1_{\fm}(\mathcal{M}),\id\right) @>>>{\vpl}_n\left(\HH^1_{\fm}(\frac{\mathcal{M}}{x^n\mathcal{M}}),\pi_n\right) @>>>{\vpl}_n\left(\HH^2_{\fm}(\mathcal{M}),zero\right) @>>>0.
\end{CD}$$Since ${\vpl}_n\left(\HH^2_{\fm}(\mathcal{M}),zero\right) =0$ we get that $$\HH^1_{\fm}(\mathcal{M})\simeq{\vpl}_n\left(\HH^1_{\fm}(\mathcal{M}),\id\right) \simeq{\vpl}_n\left(\HH^1_{\fm}(\frac{\mathcal{M}}{x^n\mathcal{M}}),\pi_n\right).$$
Similarly, $\HH^0_{\fm}(\mathcal{M}) \simeq{\vpl}_n\HH^0_{\fm}(\frac{\mathcal{M}}{x^n\mathcal{M}})$.  This is zero, since $\depth(\mathcal{M})>0$. Now, we put these in the following
diagram:$$
\begin{CD}
0 ={\vpl}_n\Gamma_{\fm}(\frac{\mathcal{M}}{ \fa^n\mathcal{M}}) @>>>{\vpl}_n\frac{\mathcal{M}}{ \fa^n\mathcal{M}} @>>>{\vpl}_n\HH^0(\mathcal{X},\frac{\mathcal{F}}{ \fa^n\mathcal{F}}) @>>>{\vpl}_n\HH^1_{\fm}(\frac{\mathcal{M}}{ \fa^n\mathcal{M}})@>>>0\\
 @. \simeq @  AAA @AAA\simeq @AAA   \\
0 =\Gamma_{\fm}(\mathcal{M}) @>>>\mathcal{M} @>>>\HH^0(\mathcal{X},\mathcal{F}) @>>>\HH^1_{\fm}(\mathcal{M})@>>>0.\\
\end{CD}
$$We apply 5-lemma to deduce that $\HH^0(\widehat{\mathcal{X}},\widehat{\mathcal{F}})\simeq{\vpl}_n\HH^0(\mathcal{X},\frac{\mathcal{F}}{ \fa^n\mathcal{F}})\simeq\HH^0(\mathcal{X},\mathcal{F}).$
\end{proof}

This implies that:
\begin{corollary}\label{GCM1} Let $(\mathcal{R},\fm)$ be a  generalized Cohen-Macaulay complete local  ring of dimension $>2$ and $\fa=(x)$ for some regular element $x$.
Then
$\mathcal{Y}$ is $\Gg$ in $\mathcal{X}$.
\end{corollary}
Corollary \ref{GCM1} is not true in 2-dimensional case, even the ring is regular (see Corollary \ref{ng1}).

 \begin{corollary}\label{}Let $(\mathcal{R},\fm)$ be a complete local ring. If $\depth(\frac{\mathcal{R}}{\fa^n})>1$ for all $n\gg 0$, then
$\mathcal{Y}$ is $\Gg$ in $\mathcal{X}$. In particular, suppose $\depth(\mathcal{R})>2$ and  $\fa$ is generated by a regular element, then
$\mathcal{Y}$ is $\Gg$ in $\mathcal{X}$.
\end{corollary}

\begin{proof}
Recall that $\Gamma_{\fm}(\mathcal{R}/ \fa^n)=\HH^1_{\fm}(\mathcal{R}/ \fa^n)=0$ for $n\gg0$.  In view of the following  exact sequence $$0\lo\Gamma_{\fm}(\mathcal{R}/ \fa^n)
\lo \mathcal{R}/ \fa^n\lo \HH^0(\mathcal{X},\widetilde{\mathcal{R}/ \fa^n})\lo \HH^1_{\fm}(\mathcal{R}/ \fa^n)\lo 0,$$we see
$\HH^0(\mathcal{X},\widetilde{\mathcal{R}/ \fa^n})\simeq \mathcal{R}/ \fa^n $ for $n\gg0$.  Therefore, $$\HH^0(\widehat{\mathcal{X}},\mathcal{O}_{\widehat{\mathcal{X}}})\simeq
{\vpl}_n \HH^0(\mathcal{X},\widetilde{\mathcal{R}/ \fa^n})\simeq{\vpl}_n (\frac{\mathcal{R}}{ \fa^n}) =\mathcal{R}^{\wedge_{\fa}}\stackrel{(\ast)}\simeq\mathcal{R},$$
where $(\ast)$ follows by the following easy observation. Suppose  a module is complete with respect to $\fm$-adic topology, then
it is  complete  with respect to $\fa$-adic topology too.
\end{proof}


\end{document}